\documentclass{amsproc}
\usepackage{a4wide}
\usepackage{color}
\usepackage{amssymb}

\newtheorem{theorem}{Theorem}[section]
\newtheorem{lemma}[theorem]{Lemma}
\newtheorem{proposition}[theorem]{Proposition}
\newtheorem{corollary}[theorem]{Corollary}

\theoremstyle{definition}
\newtheorem{definition}[theorem]{Definition}

\theoremstyle{remark}
\newtheorem{remark}[theorem]{Remark}

\numberwithin{equation}{section}

\usepackage{epsfig}
\usepackage[all]{xy}

\newcommand{\C}{{\mathbb{C}}}

\newcommand{\Z}{{\mathbb{Z}}}

\newcommand{\R}{{\mathbb{R}}}

\newcommand{\algt}{\mathfrak{t}}

\begin{document}

\title{An Orlik-Raymond type classification of simply connected six-dimensional torus manifolds with vanishing odd degree cohomology}
\date{\today}

\author[S.\ Kuroki]{Shintar\^o Kuroki}
\dedicatory{Dedicated to Professor Mikiya Masuda on his 60th birthday.}
\address{
The University of Tokyo 
3-8-1 Komaba Meguro-ku Tokyo 153-8914, JAPAN  \\
              Tel.: +81-3-5465-7001\\
              Fax: +81-3-5465-7011\\}
              \email{kuroki@ms.u-tokyo.ac.jp}     
\thanks{The author was partially supported by 
Grant-in-Aid for Scientific Research (S)
24224002, Japan Society for Promotion of Science, and 
the JSPS Strategic Young Researcher Overseas Visits Program for Accelerating Brain Circulation
"Deepening and Evolution of Mathematics and Physics, Building of International Network Hub based on OCAMI".}



\maketitle


\begin{abstract}
The aim of this paper is to classify simply connected $6$-dimensional torus manifolds with vanishing odd degree cohomology.
It is shown that there is a one-to-one correspondence between equivariant diffeomorphism types of these manifolds 
and $3$-valent labelled graphs, called torus graphs introduced by Maeda-Masuda-Panov.
Using this correspondence and combinatorial arguments, we prove that 
a simply connected $6$-dimensional torus manifold with $H^{odd}(M)=0$ is equivariantly diffeomorphic to 
the $6$-dimensional sphere $S^{6}$ or an equivariant connected sum of copies of 
$6$-dimensional quasitoric manifolds or $S^{4}$-bundles over $S^{2}$.
\end{abstract}

\section{Introduction}
\label{intro}
Let $M$ be a $2n$-dimensional closed, connected, oriented manifold with an effective $n$-dimensional (i.e., half-dimensional) torus $T^{n}$-action.
We call $M$ (or $(M,T)$) a {\it torus manifold} if $M^{T}\not=\emptyset$ (see \cite{HaMa}), 
where $M^{T}$ is the set of fixed points.
A toric manifold (i.e., a non-singular, complete toric variety viewed as a complex analytic space) with the restricted $T^{n}$-action
is a typical example of torus manifolds. 
Recall that a toric manifold is a complex $(\C^{*})^{n}$-manifold with the dense orbit (see \cite{Od}, \cite{Fu}), and $T^{n}$ is the maximal compact subgroup of $(\C^{*})^{n}$.
A fundamental result of toric geometry tells us that there is 
a one-to-one correspondence between toric manifolds and combinatorial objects called fans. 
Thus, topological (more precisely, geometric) invariants of toric manifolds can be described in terms of combinatorial invarinats of fans,
such as equivariant cohomology rings, equivariant characteristic classes and other topological invariants.

In 2003, Hattori-Masuda introduced a torus manifold as the topological generalization of a toric manifold in \cite{HaMa}.
They also introduced the combinatorial objects, called multi-fans (see \cite{Ma99}, \cite{HaMa}), and computed   
topological invariants (such as equivariant characteristic classes or Todd genus for unitary torus manifolds) in terms of multi-fans.
However, unlike toric geometry, a multi-fan does not contain enough information to 
determine some topological invariants of torus manifolds (e.g. equivariant cohomology).
So, in 2007, Maeda-Masuda-Panov introduced another combinatorial objects, called torus graphs motivated by GKM graphs introduced by Guillemin-Zara in \cite{GuZa}.
The combinatorial information of torus graphs can completely determine 
the equivariant cohomology rings of torus manifolds with vanishing odd degree cohomology, i.e., $H^{odd}(M;\Z)=0$ 
(in this paper, we only consider the integer coefficient), 
see \cite{MaPa}, \cite{MMP} (and also see Section \ref{sect3} in this paper about torus graphs).
However, in general, 
there is no one-to-one correspondence between torus manifolds with $H^{odd}(M)=0$ and torus graphs.

So, we are naturally led to ask the following two questions:
(1) which subclasses of torus manifolds are completely determined by combinatorial objects (like multi-fans or torus graphs);
(2) if we find such a subclass of torus manifolds, how we can classify such torus manifolds.
Several mathematicians have answered to the 1st question; for example, 
Davis-Januszkiewicz \cite{DaJa} for the subclass called {\it quasitoric manifolds} (see \cite{BuPa} or Section \ref{sect4.3} in this paper), 
Ishida-Fukukawa-Masuda \cite{IFM} for the subclass called {\it topological toric manifolds},
and Wiemeler for the class of simply connected $6$-dimensional torus manifolds with $H^{odd}(M)=0$   
in \cite{Wi2} (see Theorem \ref{Wi2}).
The aim of this paper is to answer to the 2nd question 
for the class of simply connected $6$-dimensional torus manifolds with $H^{odd}(M)=0$ by using torus graphs.

Let us briefly recall the classification results for torus manifolds with lower dimensions.
If $T^{1}$ acts on a compact $2$-dimensional manifold $M$, then $M$ is the $2$-dimensional sphere $S^{2}$, 
the $2$-dimensional real projectiove space $\R P^{2}$ or the $2$-dimensional torus $T^{2}$.
Because $M^{T}\not=\emptyset$ and $M$ is oriented,
$M$ must be equivariantly diffeomorphic to $S^{2}$ with $T^{1}$-action (also see \cite{Ka}).
When $\dim M=4$. 
By Orlik-Raymond's theorem in \cite{OrRa}, we have the following fact:
\begin{theorem}[Orlik-Raymond]
\label{OR-thm}
Let $M$ be a $4$-dimensional simply connected torus manifold.
Then, $M$ is equivariantly diffeomorphic to the $4$-sphere $S^4$ or an equivariant connected sum of copies of complex projective spaces $\C P^2$, $\overline{\C P}^2$  (reversed  orientation) or Hirzebruch surfaces $H_{k}$.
\end{theorem}
Here, in Theorem \ref{OR-thm}, a Hirzebruch surface $H_{k}$ ($k\in \Z$) is a manifold which is defined by the projectivization of 
the complex $2$-dimensional vector bundle $\gamma^{\otimes k}\oplus\epsilon$ over $\C P^1$, 
where $\gamma$ is the tautological and $\epsilon$ is the trivial complex line bundles over $\C P^1$.

In this paper, we prove that 
an Orlik-Raymond type theorem in Theorem \ref{OR-thm} also holds for simply connected 
$6$-dimensional torus manifolds with $H^{odd}(M)=0$. 
Before we state our main results, we introduce the result for non-simply connected torus manifolds.
As one of the consequences of Masuda-Panov's theorem (see Theorem \ref{2-equiv} in Section \ref{sect2.2}), 
we have the following proposition (also see \cite{Wi2}):
\begin{proposition}
\label{wiemeler}
Let $W$ be a $6$-dimensional torus manifold with $H^{odd}(W)=0$ (might not be simply connected).
Then, there are a simply connected, $6$-dimensional torus manifold $M$ with $H^{odd}(M)=0$ 
and a homology $3$-sphere $hS^{3}$ such that 
\[
W\cong M\#_{T}(hS^3\times T^3)
\]
up to equivariantly diffeomorphism.
\end{proposition}
Here, in Proposition \ref{wiemeler}, the product manifold 
$hS^3\times T^3$ is the product of $hS^{3}$ and the $3$-dimensional torus $T^{3}$ with the free $T^3$-action on the $2$nd factor, 
and the symbol $\#_{T}$ represents the equivariant gluing along two free orbits of $M$ and $hS^3\times T^3$.
 
Because the fundamental groups $\pi_{1}(W)$ and $\pi_{1}(hS^{3})$ are isomorphic in Proposition \ref{wiemeler},
$W$ is simply connected if and only if $hS^{3}$ is simply connected, i.e., the standard sphere.
Our main theorem is a classification of simply connected torus manifolds appeared in Proposition \ref{wiemeler}.
\begin{theorem}
\label{main1}
Let $M$ be a simply connected $6$-dimensional torus manifold with $H^{odd}(M)=0$.
Then, $M$ is equivariantly diffeomorphic to 
the $6$-sphere $S^6$ or obtained by an equivariant connected sum of copies of $6$-dimensional quasitoric manifolds 
or $S^4$-bundles over $S^2$ equipped with the structure of a torus manifold.
\end{theorem}
This type of classification, i.e., classification by equivariant connected sum, 
may be regarded as the $6$-dimensional analogue of Orlik-Raymond's classification in Theorem \ref{OR-thm}. 
So, in this paper, we call this theorem an {\it Orlik-Raymond type classification} (also see the papers \cite{Mc} and \cite{Ku08}).

\begin{remark}
In the paper \cite{Iz}, Izmestiev proves an Orlik-Raymond type classification for some class of $3$-dimensional small covers 
(i.e., the real analogue of quasitoric manifolds, see Section \ref{sect4.2}), called a linear model.
\end{remark}

The organization of this paper is as follows.
In Section \ref{sect2}, we recall the basic facts about torus manifolds. 
In Section \ref{sect3}, we recall a torus graph.
In particular, Corollary \ref{key2} is the key fact to prove Theorem \ref{main1}.
In Section \ref{sect4}, we introduce the torus graphs of  
 $S^{6}$, quasitoric manifolds and $S^{4}$-bundles over $S^{2}$.
These torus graphs will be the basic graphs to classify simply connected $6$-dimensional torus manifolds with $H^{odd}(M)=0$. 
In Section \ref{sect5}, we introduce the ``oriented'' torus graphs and 
translate the equivariant connected sum around fixed points of torus manifolds to the connected sum around vertices of oriented torus graphs.
In Section \ref{sect6} and \ref{sect7}, we prove Theorem \ref{main1}.
The brief outline of the proof is as follows.
Due to Corollary \ref{key2}, there is a one-to-one corresponds between $6$-dimensional simply connected torus manifolds with $H^{odd}(M)=0$
and $3$-valent torus graphs.
Therefore, in order to prove Theorem \ref{main1}, 
it is enough to prove that an (oriented) torus graph can be decomposed into basic torus graphs in Section \ref{sect4} 
by the connected sum.
We prove this by using combinatorial arguments.

\section{Orbit spaces of torus manifolds}
\label{sect2}

In this section, we recall some basic facts about torus manifolds (see \cite{Ma99} or \cite{HaMa} for details).

\subsection{Torus manifolds}
\label{sect2.1}

A $2n$-dimensional torus manifold $M$ is said to be {\it locally standard}
if every point in $M$ has a $T$-invariant open neighborhood $U$ which is
weakly equivariantly homeomorphic to an open subset $\Omega_{U}\subset \C^{n}$
invariant under the standard $T^{n}$-action on $\C^{n}$, 
where two group actions $(U,T)$ and $(\Omega_{U},T)$ are said to be {\it weakly equivariantly homeomorphic} 
if  
there is an equivariant homeomorphism from $U$ to $\Omega_{U}$ up to an automorphism on $T^{n}$ (see e.g. \cite[Section 2.1]{Ku11} for details).  

Let $M_{i}$, $i=1,\ \ldots,\ m$, be a codimension-two torus submanifold in a $2n$-dimensional torus manifold $M$ 
which is fixed by some circle subgroup $T_{i}$ in $T$.
Such $M_{i}$ is a $(2n-2)$-dimensional torus manifold with $T/T_{i}$-action, 
called a {\it characteristic submanifold}.
Because a torus manifold $M$ is compact, the cardinality of all characteristic submanifolds in $M$ is finite.
If $M$ is locally standard, each characteristic submanifold is also locally standard.

An {\it omniorientation} $\mathcal{O}$ of $M$ is 
a choice of orientation for the torus manifold $M$ 
as well as for each characteristic submanifold.
If there are just $m$ characteristic submanifolds in $M$, 
there are exactly $2^{m+1}$ omniorientations (see \cite{BuPa}, \cite{HaMa}).
If $M$ has a $T$-invariant almost complex structure $J$ (in this case, $M$ is automatically locally standard), 
then there exists the canonical omniorientation $\mathcal{O}_{J}$ determined by $J$.
We call the torus manifold $M$ with a fixed omniorientation $\mathcal{O}$ an {\it omnioriented torus manifold} and 
denote it by $(M,\mathcal{O})$.

\subsection{Orbit spaces of locally standard torus manifolds}
\label{sect2.2}

The orbit space $M/T$ of a locally standard torus manifold $M$ naturally admits the structure of a ``topological'' manifold with corners.
We next recall the basic facts of a topological manifold with corners 
(cf. the definition of a smooth manifold with corners in \cite{Le}) and introduce the structure on $M/T$.

The following notations will be often used:
\begin{eqnarray*}
[n]=\{0,\ 1,\ \ldots,\ n\},
\end{eqnarray*}
and 
\begin{eqnarray*}
\R^{n}_{+}=\{(x_{1},\ldots,x_{n})\in \R^{n}\ |\ x_{i}\ge 0,\ i=1,\ldots,n\}.
\end{eqnarray*}
Let $Q^{n}$ be an $n$-dimensional topological manifold with boundary.
A {\it chart with corners} for $Q^{n}$ is a pair $(V,\psi_{V})$, where 
$V$ is an open subset of $Q^{n}$ and 
\[
\psi_{V}:V\to \R^{n}_{+}
\]
is homeomorphic from $V$ to a (relatively) open subset $\Omega_{V}\subset \R^{n}_{+}$.
Two charts with corners $(V,\psi_{V}),\ (W,\psi_{W})$ are said to be {\it (topologically) compatible} if the
composition of functions $\psi_{V}\circ \psi_{W}^{-1}:\psi_{W}(V\cap W)\to \psi_{V}(V\cap W)$ 
is a strata-preserving homeomorphism. 
This implies that if $\psi_{W}(p)\in \R^{n}_{+}$ contains exactly $k$ zero coordinates then $\psi_{V}(p)\in \R^{n}_{+}$ also contains exactly $k$ zero coordinates for $0\le k\le n$.
We call the collection of compatible charts with corners $\{(V,\psi_{V})\}$ whose domains cover $Q^{n}$ an {\it atlas}.
Then, its maximal atlas is called a {\it structure with corners} of $Q^{n}$.
A topological manifold with boundary together with a structure with corners is called 
a {\it (topological) manifold with corners}.
Let $p\in Q^{n}$ be a point of an $n$-dimensional manifold with corners $Q^{n}$.
A chart $(V,\psi_{V})$ with corners such that $p\in V$ defines the number $d(p)\in [n]$ 
by the number of zero-coordinates of $\psi_{V}(p)\in \R^{n}_{+}$.
By the compatibility of charts, this number is independent on the choice of a chart with corners which contains $p$.
Therefore, the map $d:Q^{n}\to [n]$ is well-defined. 
The number $d(p)$ is called a {\it depth} of $p$.
We call the closure of a connected component of $d^{-1}(k)$ ($0\le k\le n$) a {\it codimension-$k$ face}. 
In particular, the codimension-$0$ face is $Q^{n}$ itself. 
Moreover,  
codimension-$1$, $(n-1)$ and $n$ faces are called {\it facet}, {\it edge} and {\it vertex}, respectively.
The set of all edges and vertices is called a {\it one-skeleton of $Q^{n}$} (or a {\it graph of $Q^{n}$}).
By restricting the structure with corners on $Q^{n}$ to faces, 
we may regard each codimension-$k$ face as an $(n-k)$-dimensional (sub)manifold with corners.
\begin{definition}[Manifold with faces]
An $n$-dimensional manifold with corners $Q$ is said to be a {\it manifold with faces} (or a {\it nice manifold with corners})
if $Q$ satisfies the following conditions:
\begin{enumerate}
\item for every $k\in [n]$, there exists a codimension-$k$ face;
\item for each codimension-$k$ face $H$, there are exactly $k$ facets $F_{1},\ \ldots,\ F_{k}$ 
such that $H$ is a connected component of $\cap_{i=1}^{k}F_{i}$; moreover, $H\cap F\not=H$ 
for any facet $F\not=F_{i}$ ($i=1,\ldots,k$).
\end{enumerate} 
\end{definition}

Let $(M,T)$ be a torus manifold.
When $(M,T)$ is locally standard,
by using the differentiable slice theorem,
the orbit space $M/T$ has the structure of an $n$-dimensional manifold with faces.
On the other hand, when $M$ satisfies $H^{odd}(M)=0$, 
its orbit space $M/T$ satisfies a stronger condition by
the following Masuda-Panov theorem (see \cite[Lemma 2.1]{MaPa} and \cite[Theorem 2]{MaPa}):
\begin{theorem}[Masuda-Panov]
\label{2-equiv}
Let $M$ be a $2n$-dimensional torus manifold.
Then, the following two conditions are equivalent:
\begin{enumerate}
\item $H^{odd}(M)=0$;
\item the $T$-action on $M$ is locally standard and its orbit space $M/T$ has 
the structure of an $n$-dimensional face acyclic manifold with corners.
\end{enumerate}
\end{theorem}
Here, in Theorem \ref{2-equiv}, an $n$-dimensional {\it face acyclic} manifold with corners $Q$ is 
an $n$-dimensional manifold with faces such that 
all faces $F$ (include $Q$) of $Q$ are acyclic, i.e., $H_{*}(F)\simeq H_{0}(F)\simeq \Z$. 
For example, if $Q$ is a simply connected $3$-dimensional face acyclic manifold with corners, then it is easy to check that 
the boundary of $Q$ is homeomorphic to the $2$-sphere $S^{2}$.
Moreover, in this case, we can also check that $Q$ itself is homeomorphic to the $3$-dimensional disk $D^{3}$.
Therefore, as one of the conclusions of Theorem \ref{2-equiv}, we have the following corollary:
\begin{corollary}
\label{disk}
Let $M$ be a simply connected $6$-dimensional torus manifold with $H^{odd}(M)=0$.
Then, its orbit space $M/T$ is homeomorphic to the $3$-dimensional disk.
\end{corollary}

By the definition of a manifold with faces $Q$, 
we can define a simplicial poset (partially ordered set) $\mathcal{P}(Q)$, called a {\it face poset} of $Q$ (see \cite{Ma05}),
by the set of faces in $Q$ with the empty-set $\emptyset$ ordered by inclusion,
where the empty set $\emptyset$ is the smallest element by this order, say $\preceq$.
We often denote the face poset structure of $Q$ as $(\mathcal{P}(Q), \preceq)$.
Let $Q_{1}$ and $Q_{2}$ be $n$-dimensional manifolds with faces.
We call $Q_{1}$ and $Q_{2}$ are {\it combinatorially equivalent} if their 
face posets $(\mathcal{P}(Q_{1}), \preceq_{1})$ and $(\mathcal{P}(Q_{2}), \preceq_{2})$ are isomorphic as a poset (i.e., 
there is an order-preserving bijection between them).
We denote them by $Q_{1}\approx_{c}Q_{2}$.
By definition of weakly equivariant homeomorphism, if two locally standard torus manifolds $M_{1}$ and $M_{2}$ are weakly equivariantly homeomorphic then $M_{1}/T\approx_{c}M_{2}/T$.

\subsection{Characteristic functions}
\label{sect2.3}

Let $M$ be a $2n$-dimensional locally standard torus manifold.
By the argument demonstrated in Section \ref{sect2.2}, 
the orbit map $\pi:M\to M/T=Q$ may be regarded as the projection onto some manifold with faces $Q$.
Let $\mathcal{F}(Q)=\{F_{1},\ \ldots,\ F_{m}\}\subset \mathcal{P}(Q)$ be the set of all facets in $Q$.
By the definition of facet $F_{i}\in \mathcal{F}(Q)$, its preimage $\pi^{-1}(F_{i})$ is a characteristic submanifold $M_{i}$. 
Then, there exists the circle subgroup $T_{i}(\subset T)$ fixing $M_{i}=\pi^{-1}(F_{i})$ (recall that $\dim M_{i}=2n-2$).
Recall that $T_{i}$ is determined by a primitive element in $\algt_{\Z}\simeq \Z^{n}$ (the lattice of Lie algebra of $T$).
Therefore, by taking this primitive element (up to sign) in $\algt_{\Z}$, we can define the following map: 
\[
\lambda:\mathcal{F}(Q)\to \algt_{\Z}/\{\pm 1\}
\]
where $\algt_{\Z}/\{\pm 1\}$ represents the quotient of $\algt_{\Z}$ by signs.
We call $\lambda$ a {\it characteristic function}.

Now the choice of omniorientation $\mathcal{O}$ of $M$ determines the sign of $\lambda$ as follows.
Fix an omniorientation $\mathcal{O}$ of $M$.
Namely, we fix the orientation of the tangent bundle of $M$ (resp. $M_{i}$), say $\tau$ (resp. $\tau_{i}$).
Restricting $\tau$ to the submanifold $M_{i}$, say $\tau|_{M_{i}}$, we obtain the $T^{n}$-equivariant decomposition
$\tau|_{M_{i}}\simeq \tau_{i}\oplus\nu_{i}$, where $\nu_{i}$ is the $T_{i}$-equivariant normal bundle of $M_{i}$.
Therefore, because we fix the orientation of $\tau|_{M_{i}}$ (induced from the orientation of $\tau$) and that of $\tau_{i}$,
we may choose an orientation of $\nu_{i}$ such that the orientation of $\tau|_{M_{i}}$ coincides with that of $\tau_{i}\oplus \nu_{i}$ (thus, we may regard $\nu_{i}$ as the complex line bundle over $M_{i}$).
Because $T_{i}$ acts on $\nu_{i}$, 
we may choose an orientation of $T_{i}$ such that the $T_{i}$-action preserves the orientation of $\nu_{i}$.
This orientation of $T_{i}$ determines the sign of $\lambda(F_{i})$ for $i=1,\ \ldots,\ m$.  
By this way, we have the following function: 
\[
\lambda_{\mathcal{O}}:\mathcal{F}(Q)\to \algt_{\Z}.
\]
This is called an {\it omnioriented characteristic function} (of $(M,\mathcal{O})$), in this paper.

\begin{remark}
\label{rem-chfct}
The characteristic function defined in \cite{Wi2} may be regarded as the characteristic function $\lambda$ as above.
On the other hand, the characteristic function defined in \cite{DaJa} may be regarded as the characteristic function $\lambda_{\mathcal{O}}$ as above
by taking an appropriate omniorientation (also see \cite[Section 5.2]{BuPa}).
\end{remark}

Let $p\in M^{T}$.
We define the subset $I_{p}\subset [m]$ as follows:
\[
I_{p}=\{i\in [m]\ |\ p\in M_{i}\}.
\]
By the differentiable slice theorem around $p\in M^{T}$, we have that its cardinality $|I_{p}|=n$ for every $p\in M^{T}$.
Put $I_{p}=\{i_{1},\ldots,i_{n}\}$.
Because the $T$-action on $M$ is effective, 
$\{\lambda(F_{i_{1}}),\ldots,\lambda(F_{i_{n}})\}$ spans $\algt^{*}_{\Z}/\{\pm 1\}$, i.e.,  
the determinant of the induced $(n\times n)$-matrix 
\[
( \lambda(F_{i_{1}}) \cdots \lambda(F_{i_{n}}))
\] 
satisfies that
\begin{eqnarray}
\label{smooth-condition}
\det \left( \lambda(F_{i_{1}}) \cdots \lambda(F_{i_{n}}) \right)=\pm 1.
\end{eqnarray}
Similarly, we have 
\begin{eqnarray}
\label{smooth-condition2}
\det \left( \lambda_{\mathcal{O}}(F_{i_{1}}) \cdots \lambda_{\mathcal{O}}(F_{i_{n}}) \right)=\pm 1.
\end{eqnarray}
for each $n$ facets such that $\cap_{j=1}^{n} F_{i_{j}}=\{p\}$ for some vertex $p\in Q$ (called {\it the facets around a vertex}).

Motivated by the observations as above, 
we may abstractly define the characteristic function on a manifold with faces as follows (see \cite{BuPa}, \cite{DaJa} for simple polytopes and \cite{MaPa}, \cite{Wi2} for manifold with faces):

\begin{definition}
Let $Q$ be an $n$-dimensional manifold with faces and $\mathcal{F}(Q)$ be the set of its facets.
Let $\algt_{\Z}$ be the lattice of Lie algebra of $T^{n}$ and $\algt_{\Z}/\{\pm 1\}$ be its quotient by $\{\pm 1\}$.
Then, a function $\lambda:\mathcal{F}(Q)\to \algt_{\Z}/\{\pm 1\}$ is said to be 
a {\it characteristic function} if $\lambda$ satisfies the relation \eqref{smooth-condition} for the facets around every vertex, and a function $\lambda_{\mathcal{O}}:\mathcal{F}(Q)\to \algt_{\Z}$ is said to be an 
{\it omnioriented characteristic function} if $\lambda_{\mathcal{O}}$ satisfies the relation \eqref{smooth-condition2} for the facets around every  vertex.
\end{definition}

We denote an $n$-dimensional manifold with faces $Q$ with its characteristic function $\lambda$ (resp. omnioriented characteristic function $\lambda_{\mathcal{O}}$) by $(Q,\lambda)$ (resp. $(Q,\lambda_{\mathcal{O}})$).

Let $Q_{1}$ and $Q_{2}$ be manifolds with faces and $\lambda_{1}$ (resp. $\lambda_{\mathcal{O}_{1}}$) and $\lambda_{2}$ (resp. $\lambda_{\mathcal{O}_{2}}$) be their (resp. omnioriented) characteristic functions, respectively. 
Assume that $Q_{1}\approx_{c} Q_{2}$ and it is induced by the bijective map $\widetilde{f}:\mathcal{P}(Q_{1})\to \mathcal{P}(Q_{2})$.
Denote its restriction onto the set of facets as 
\[
f=\widetilde{f}|_{\mathcal{F}(Q_{1})}:\mathcal{F}(Q_{1})\to \mathcal{F}(Q_{2}).
\] 
Then, we call that $(Q_{1},\lambda_{1})$ and $(Q_{2},\lambda_{2})$ are {\it combinatorially equivalent} if the following diagram commutes:
\[
\xymatrix{
& \mathcal{F}(Q_{1}) \ar[d]^{f} \ar[r]^{\lambda_{1}} & \algt_{\Z}/\{\pm 1\} \ar[d]^{Id} \\
& \mathcal{F}(Q_{2}) \ar[r]^{\lambda_{2}} & \algt_{\Z}/\{\pm 1\}
}
\]
Similarly, $(Q_{1},\lambda_{\mathcal{O}_{1}})$ and $(Q_{2},\lambda_{\mathcal{O}_{2}})$ are {\it combinatorially equivalent} if the following diagram commutes:
\[
\xymatrix{
& \mathcal{F}(Q_{1}) \ar[d]^{f} \ar[r]^{\quad \lambda_{\mathcal{O}_{1}}} & \algt_{\Z} \ar[d]^{Id} \\
& \mathcal{F}(Q_{2}) \ar[r]^{\quad \lambda_{\mathcal{O}_{2}}} & \algt_{\Z}
}
\]

Note that the characteristic function $\lambda$ can be obtained by ignoring sings from the omnioriented characteristic function  $\lambda_{\mathcal{O}}$; we call such $\lambda$ an {\it induced characteristic function} from $\lambda_{\mathcal{O}}$.
On the other hand, by choosing a sign for each facet, we can obtain an 
omnioriented characteristic function  $\lambda_{\mathcal{O}}$ from the characteristic function $\lambda$;
we call such $\lambda_{\mathcal{O}}$ an {\it induced oriented characteristic function} from $\lambda$.
Therefore, we have the following lemma:
\begin{lemma}
\label{omni->nonomni}
If $(Q_{1},\lambda_{\mathcal{O}_{1}})$ and $(Q_{2},\lambda_{\mathcal{O}_{2}})$ are combinatorially equivalent, then
their induced $(Q_{1},\lambda_{1})$ and $(Q_{2},\lambda_{2})$ are also combinatorially equivalent.

If $(Q_{1},\lambda_{1})$ and $(Q_{2},\lambda_{2})$ are combinatorially equivalent, then
there are induced omnioriented characteristic functions $\lambda_{\mathcal{O}_{1}}$ and $\lambda_{\mathcal{O}_{2}}$ such that $(Q_{1},\lambda_{\mathcal{O}_{1}})$ and $(Q_{2},\lambda_{\mathcal{O}_{2}})$ are combinatorially equivalent.
\end{lemma}

Now we may introduce one of the key facts to prove our main theorem (see \cite[Theorem 1.3, Theorem 6.1]{Wi2}):

\begin{theorem}[Wiemeler]
\label{Wi2}
Let $M_{1}$ and $M_{2}$ be a $6$-dimensional simply connected torus manifold with $H^{odd}(M)=0$
and $(Q_{1},\lambda_{1})$ and $(Q_{2},\lambda_{2})$ be their orbit spaces 
with characteristic functions.
Then, the following three statements are equivalent:
\begin{enumerate}
\item $(Q_{1},\lambda_{1})$ and $(Q_{2},\lambda_{2})$ are combinatorially equivalent;
\item $M_{1}$ and $M_{2}$ are equivariantly homeomorphic; 
\item $M_{1}$ and $M_{2}$ are equivariantly diffeomorphic.
\end{enumerate}
\end{theorem}

Therefore, by Corollary \ref{disk} and Theorem \ref{Wi2}, 
in order to classify all $6$-dimensional simply connected torus manifolds with $H^{odd}(M)=0$,
it is enough to classify all $(Q,\lambda)$'s up to combinatorial equivalent,
where $Q$ is a $3$-dimensional disk equipped with the structure of a manifold with faces.

\section{Torus graph induced from manifold with faces}
\label{sect3}

Let $(M,\mathcal{O})$ be an omnioriented locally standard $2n$-dimensional torus manifold and 
$(Q,\lambda_{\mathcal{O}})$ be its orbit space with an omnioriented characteristic function.
From the one-skeleton of $(Q,\lambda_{\mathcal{O}})$, we can define a labelled graph, called a {\it torus graph}.
One of the key arguments to prove the main theorem is to classify all possible torus graphs (see Section \ref{sect7}).
To do that, in this section, we recall a torus graph defined by Maeda-Masuda-Panov \cite{MMP}.

Let $\Gamma$ be the graph of $Q$.
Let $V(\Gamma)$ be its vertices and $E(\Gamma)$ be its oriented edges, i.e.,
we distinguish two edges $pq$ and $qp$. 
For $p\in V(\Gamma)$, we denote the set of outgoing edges from $p$ as $E_{p}(\Gamma)$.
Because $Q$ is an $n$-dimensional manifold with faces, 
$|E_{p}(\Gamma)|=n$ and 
each edge $e\in E(\Gamma)$ is a connected component of $\cap_{i=1}^{n-1}F_{i}$,  
for some $F_{1},\ F_{2},\ldots, F_{n-1}\in \mathcal{F}(Q)$.
Moreover, for $p\in V(\Gamma)$ which is one of two vertices on $e$, 
there is another facet $F_{n}\in \mathcal{F}(Q)$ such that 
$\{p\}$ is a connected component of $\cap_{i=1}^{n}F_{i}$.
In other words, $F_{n}$ may be regarded as a normal facet of $e\in E(\Gamma)$ on $p\in V(\Gamma)$.
Put $\lambda_{\mathcal{O}}(F_{i})=a_{i}\in \algt_{\Z}\simeq \Z^{n}$.
Then, there exists unique $\alpha\in \algt_{\Z}^{*}$ such that 
\begin{eqnarray}
\label{dual}
\langle \alpha,a_{i} \rangle=0\ {\rm for}\ i=1,\ldots, n-1\quad {\rm and} \quad \langle \alpha,a_{n} \rangle=+1,
\end{eqnarray}
where $\langle , \rangle$ represents the pairing of $\algt^{*}$ and $\algt$.
Therefore, by this way, we can define a map $\mathcal{A}:E(\Gamma)\to \algt_{\Z}^{*}$ from the omnioriented characteristic function $\lambda_{\mathcal{O}}$.
This map $\mathcal{A}$ is called an {\it axial function} on $\Gamma$.
We call the labelled graph $(\Gamma,\mathcal{A})$ a {\it torus graph} induced from $(Q,\lambda_{\mathcal{O}})$ (or equivalently $(M,\mathcal{O})$).
We denote such a torus graph as $\Gamma(Q,\lambda_{\mathcal{O}})$ (or $(\Gamma_{M},\mathcal{A}_{M})$).
We can easily check the following proposition by definition of torus graphs (also see \cite{MMP}):
\begin{proposition}
\label{connection}
Let $(\Gamma,\mathcal{A})$ be a torus graph induced from $(Q,\lambda_{\mathcal{O}})$.
Then, $\Gamma$ is an $n$-valent regular graph, i.e., $|E_{p}(\Gamma)|=n$ for all $p\in V(\Gamma)$, 
and $(\Gamma,\mathcal{A})$ satisfies the following conditions:
\begin{enumerate}
\item $\mathcal{A}(e)=\pm \mathcal{A}(\bar{e})$, where $\bar{e}$ is the orientation reversed edge of $e$;
\item $\{\mathcal{A}(e)\ |\ e\in E_{p}(\Gamma)\}$ spans $\algt_{\Z}^{*}$ for all vertices $p\in V(\Gamma)$;
\item there is a bijection $\nabla_{pq}:E_{p}(\Gamma)\to E_{q}(\Gamma)$ for all edges whose initial vertex is $p$ and terminal vertex is $q$ such that the following conditions hold:
\begin{enumerate}
\item $\nabla_{\bar{e}}=\nabla_{e}^{-1}$;
\item $\nabla_{e}(e)=\bar{e}$;
\item $\mathcal{A}(\nabla_{pq}(e))-\mathcal{A}(e)\equiv 0\mod \mathcal{A}(pq)$ for all $e\in E_{p}(\Gamma)$.
\end{enumerate}
\end{enumerate}
\end{proposition}
We call $\nabla=\{\nabla_{e}\ |\ e\in E(\Gamma)\}$ a {\it connection} on $(\Gamma,\mathcal{A})$.

\begin{remark}
The original torus graph (induced from an omnioriented torus manifold)
is defined by using the {\it tangential representations}, see \cite{MaPa}, \cite{MMP}.
The definition of torus graph as above is essentially same with this definition.

In \cite{MMP}, 
motivated by the GKM graph introduced by Guillemin-Zara in \cite{GuZa},
an $n$-valent graph $\Gamma$ with a label $\mathcal{A}:E(\Gamma)\to \algt_{\Z}^{*}$ which satisfies three conditions in Proposition \ref{connection} is called an {\it (abstract) torus graph} (i.e., there might be no geometric objects which define $(\Gamma,\mathcal{A})$).
\end{remark}

We next define the equivalence relation between two torus graphs.
We call the map $f:\Gamma_{1}=(V(\Gamma_{1}),E(\Gamma_{1}))\to \Gamma_{2}=(V(\Gamma_{2}),E(\Gamma_{2}))$ 
a {\it graph isomorphism}, 
if the restricted map $f|_{V}:V(\Gamma_{1})\to V(\Gamma_{2})$ and $f|_{E}:E(\Gamma_{1})\to E(\Gamma_{2})$ are bijective and
the following map commutes:
\[
\xymatrix{
& E(\Gamma_{1}) \ar[d]^{\pi_{V_{1}}} \ar[r]^{f|_{E}} & E(\Gamma_{2}) \ar[d]^{\pi_{V_{2}}} \\
& V(\Gamma_{1}) \ar[r]^{f|_{V}} & V(\Gamma_{2}) 
}
\]
where $\pi_{V}:E(\Gamma)\to V(\Gamma)$ is the map projecting onto the initial vertex, i.e., $\pi_{V}(pq)=p$.
In other words, the bijection $f|_{V}$ preserves the edges. 
Now we may define the equivalence relation.
\begin{definition}
\label{def_torus-graph}
Let $(\Gamma_{1},\mathcal{A}_{1})$ and $(\Gamma_{2},\mathcal{A}_{2})$ be torus graphs.
We say $(\Gamma_{1},\mathcal{A}_{1})$ and $(\Gamma_{2},\mathcal{A}_{2})$ are {\it equivalent} 
if 
there is a graph isomorphism $f:\Gamma_{1}\to \Gamma_{2}$ such that the following diagram commutes:
\[
\xymatrix{
& E(\Gamma_{1}) \ar[d]^{f|_{E}} \ar[r]^{\quad \mathcal{A}_{1}} & \algt_{\Z}^{*} \ar[d]^{Id} \\
& E(\Gamma_{2}) \ar[r]^{\quad \mathcal{A}_{2}} & \algt_{\Z}^{*} 
}
\]
\end{definition}
Assume $(\Gamma,\mathcal{A})=\Gamma(Q,\lambda_{\mathcal{O}})$.
Let $\mathcal{P}_{k}(\Gamma,\mathcal{A})$ be the set of $k$-valent torus subgraphs in $(\Gamma,\mathcal{A})$, i.e.,
$k$-valent subgraphs in $\Gamma$ closed under the connection $\nabla$,
where $-1\le k\le n$ and we define $\mathcal{P}_{-1}(\Gamma,\mathcal{A})=\{\emptyset \}$.
Then, the set 
\[
\mathcal{P}(\Gamma,\mathcal{A})=\cup_{k=-1}^{n} \mathcal{P}_{k}(\Gamma,\mathcal{A})
\]
admits the structure of a simplicial poset by inclusion (see \cite{MMP}).
We denote this structure by $(\mathcal{P}(\Gamma,\mathcal{A}), \preceq)$.
Let $\mathcal{P}(Q)$ be the face poset of $Q$ (see Section \ref{sect2.2}) and $\mathcal{P}_{k}(Q)$ be the set of all $k$-dimensional faces,
where $-1\le k\le n$ and $\mathcal{P}_{-1}(Q)=\{\emptyset \}$.
Then, each element of $\mathcal{P}_{k}(\Gamma,\mathcal{A})$ is nothing but the graph of an element in $\mathcal{P}_{k}(Q)$.
This implies that the poset $(\mathcal{P}(\Gamma,\mathcal{A}), \preceq)$ is equivalent to the poset $(\mathcal{P}(Q),\preceq)$.
Therefore, we have the following lemma:
\begin{lemma}
\label{equiv<->equiv}
The following two statements are equivalent:
\begin{enumerate}
\item two manifolds with faces with omnioriented characteristic functions $(Q_{1},\lambda_{\mathcal{O}_{1}})$ and $(Q_{2},\lambda_{\mathcal{O}_{2}})$ are combinatorially equivalent;
\item their induced torus graphs $\Gamma(Q_{1},\lambda_{\mathcal{O}_{1}})$ and $\Gamma(Q_{2},\lambda_{\mathcal{O}_{2}})$ are equivalent.
\end{enumerate}
\end{lemma}

By Lemma \ref{omni->nonomni}, Theorem \ref{Wi2} and Lemma \ref{equiv<->equiv}, we have the following corollary:
\begin{corollary}
\label{key2}
Let $(M_{1},T)$ and $(M_{2},T)$ be $6$-dimensional simply connected torus manifolds with vanishing odd degree cohomology.
Then, the following statements are equivalent:
\begin{enumerate}
\item $(M_{1},T)$ and $(M_{2},T)$ are equivariantly diffeomorphic;
\item their orbit spaces, i.e., $3$-dimensional disks with the structures of manifolds with faces, with characteristic functions 
$(M_{1}/T,\lambda_{1})$ and $(M_{2}/T,\lambda_{2})$ are combinatorially equivalent;
\item there are omnioriented characteristic functions $\lambda_{\mathcal{O}_{1}}$ and $\lambda_{\mathcal{O}_{2}}$ such that 
their induced $3$-valent torus graphs $\Gamma(M_{1}/T,\lambda_{\mathcal{O}_{1}})$ and $\Gamma(M_{2}/T,\lambda_{\mathcal{O}_{2}})$ are equivalent.
\end{enumerate}
\end{corollary}

Therefore, in order to prove our main theorem (Theorem \ref{main-1}), 
it is enough to classify all $3$-valent torus graphs $(\Gamma,\mathcal{A})$, induced from $(M,\mathcal{O})$, up to equivalent.

\section{Basic six-dimensional torus manifolds}
\label{sect4}

Let $(M,T)$ be a simply connected, $6$-dimensional torus manifold with $H^{odd}(M)=0$,
and $(\Gamma_{M},\mathcal{A}_{M})(=(\Gamma,\mathcal{A}))$ be its torus graph induced by some omniorientation.
As a preliminary to prove the main theorem (Theorem \ref{main-1}),
in this section, 
we will introduce some of basic torus graphs $(\Gamma,\mathcal{A})$ and their corresponding 
$6$-dimensional torus manifolds $(M,T)$.

\subsection{$6$-sphere}
\label{sect4.1}
Because the induced torus graphs from $(M,T)$ are $3$-valent,
if there is a $3$-multiple edge, i.e., three edges that are incident to the same two vertices, 
then it follows from Proposition \ref{connection} that 
such torus graph must be the torus graph in Figure \ref{tg-S^6}.
(We denote the torus graph in Figure \ref{tg-S^6} as $(\Gamma_{sp},\mathcal{A}_{\alpha,\beta,\gamma})$.)
\begin{figure}[h]
\begin{center}
\includegraphics[width=100pt,clip]{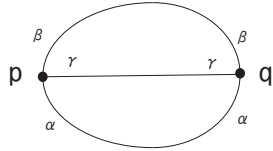}
\caption{The torus graph $(\Gamma_{sp},\mathcal{A}_{\alpha,\beta,\gamma})$, 
where $\alpha,\beta,\gamma\in \algt_{\Z}^{*}\simeq \Z^{3}$ are $\Z$-basis.}
\label{tg-S^6}
\end{center}
\end{figure}

Put $\alpha=k_{11}e_{1}+k_{12}e_{2}+k_{13}e_{3}$, $\beta=k_{21}e_{1}+k_{22}e_{2}+k_{23}e_{3}$ and 
$\gamma=k_{31}e_{1}+k_{32}e_{2}+k_{33}e_{3}$,
 by the standard basis $e_{1},e_{2},e_{3}$ in $\algt_{\Z}^{*}\simeq \Z^{3}$.
Then, the following equation holds: 
\begin{eqnarray}
\label{basis for S^6}
\det
\left(
\begin{array}{ccc}
k_{11} & k_{12} & k_{13} \\
k_{21} & k_{22} & k_{23} \\
k_{31} & k_{32} & k_{33} \\
\end{array}
\right)=\pm 1.
\end{eqnarray}

Let $S^6\subset \C^3\oplus \R$ be the unit sphere, i.e., 
the set $(z_{1},z_{2},z_{3},r)\in \C^3\oplus \R$ such that $|z_{1}|^2+|z_{2}|^2+|z_{3}|^2+r^2=1$.
Define the $T^3$-action on the first three complex coordinates in $S^{6}$ by 
\begin{eqnarray}
\label{def-action S^6}
(t_{1},t_{2},t_{3})(z_{1},z_{2},z_{3},r)\mapsto (\rho_{1}(t)z_{1},\rho_{2}(t)z_{2},\rho_{3}(t)z_{3},r)
\end{eqnarray}
where $t=(t_{1},t_{2},t_{3})\in T$ and $\rho_{i}:T\to S^1$, $i=1,2,3$, is a one-dimensional complex representation defined by 
\[
\rho_{i}(t_{1},t_{2},t_{3})=t_{1}^{k_{i1}}t_{2}^{k_{i2}}t_{3}^{k_{i3}}.
\]
Then, 
by choosing an appropriate omniorientation on $S^{6}$,
we have that its induced torus graph is equivalent to $(\Gamma_{sp},\mathcal{A}_{\alpha,\beta,\gamma})$.
Therefore, by using Corollary \ref{key2}, we have the following lemma:
\begin{lemma}
\label{S^6}
Let $(M,\mathcal{O})$ be an omnioriented $6$-dimensional simply connected torus manifold with $H^{odd}(M)=0$.
If its induced torus graph is $(\Gamma_{sp},\mathcal{A}_{\alpha,\beta,\gamma})$, then
$(M,T)$ is equivariantly diffeomorphic to one of $(S^{6},T)$ defined by \eqref{def-action S^6}.
\end{lemma}

\subsection{$S^4$-bundles over $S^2$}
\label{sect4.2}

Assume that a $3$-valent torus graph $(\Gamma,\mathcal{A})$ does not have $3$-multiple edges but have multiple edges,
i.e., two edges that are incident to the same two vertices.
In this section, we classify the easiest case of such torus graphs.

Because $\Gamma$ is a one-skeleton of $3$-dimensional manifold with faces $Q$, the number of vertices $|V(\Gamma)|\ge 4$.
Assume that $|V(\Gamma)|=4$. 
Then, we can easily check that such torus manifold is the one-skeleton of the $3$-simplex (see Figure \ref{tg-CP^3} in Section \ref{sect4.3}) 
or the graph drawn in Figure \ref{tg-bundles}, say $\Gamma_{S}$.
As is well known that the torus manifold whose torus graph is 
the one-skeleton of the $3$-simplex is equivariantly diffeomorphic to the complex projective space with some $T$-action 
(see e.g. \cite{DaJa}, and also see Figure \ref{tg-CP^3} in Section \ref{sect4.3}).
So, we only study the torus manifold which induces the graph $\Gamma_{S}$.
Because $Q$ is homeomorphic to $D^{3}$, we may regard
$Q$ whose one-skeleton is $\Gamma_{S}$ as the product of $D^{2}\times I$,
where $D^{2}$ is the $2$-dimensional disk and $I$ is the interval.
By considering all functions on facets of $Q$ which satisfies \eqref{smooth-condition2},
we can classify all omnioriented characteristic functions $\lambda_{\mathcal{O}}$ on $Q$.
Then, by the way to induce the axial function $\mathcal{A}_{S}$ from $(Q,\lambda_{\mathcal{O}})$ demonstrated in Section \ref{sect3},
we can obtain all possible axial functions on $\Gamma_{S}$ as shown in Figure \ref{tg-bundles}.

\begin{figure}[h]
\begin{center}
\includegraphics[width=150pt,clip]{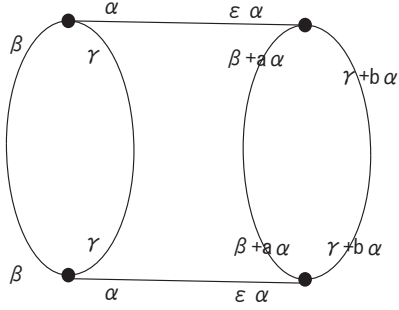}
\caption{The torus graph $(\Gamma_{S},\mathcal{A}_{S})=(\Gamma_{S},\mathcal{A}_{\alpha,\beta,\gamma}^{\epsilon,a,b})$, where $\epsilon=\pm 1$ and $a,b\in \Z$ and 
$\alpha,\beta,\gamma\in \algt_{\Z}^{*}$ are $\Z$-basis of $\algt_{\Z}^{*}$.}
\label{tg-bundles}
\end{center}
\end{figure}
 
The torus graph $(\Gamma_{S},\mathcal{A}_{S})$ in Figure \ref{tg-bundles} can be induced from an $S^4$-bundle over $S^2$ as follows.

First, by choosing $\epsilon=\pm 1$, we may define two free $T^{1}$-actions on $S^3\subset \C^2$ as follows:
\[
(w,z)\mapsto (t^{-1}w,t^\epsilon z).
\]
We denote $S^3$ with the above $T^1$-action by $S^3_{\epsilon}$. 
Note that $S^{3}_{\epsilon}/T^{1}$ is diffeomorphic to the $2$-sphere $S^2$, and a complex line bundle over $S^2$ can be denoted by 
\[
S^3_{\epsilon}\times_{T^{1}} \C_{k},
\]
where $\C_{k}$ is the complex $1$-dimensional $T^1$-representation space by $k$-times rotation for some $k\in \Z$.
Let $S^3_{\epsilon}\times_{T^1}\R$ be the trivial real line bundle over $S^2$.
Take the unit sphere bundle of the following Whitney sum of three vector bundles for $a,b\in \Z$: 
\[
S^3_{\epsilon}\times_{T^{1}} (\C_{a}\oplus \C_{b}\oplus \R).
\]
Then, we obtain the $S^{4}$-bundle over $S^{2}$ denoted by
\[
M(\epsilon,a,b)=S^3_{\epsilon}\times_{S^{1}} S(\C_{a}\oplus \C_{b}\oplus \R),
\]
for $\epsilon=\pm 1$, $a,b\in \Z$.
Namely, we can identify elements in $M(\epsilon,a,b)$ by
\[
[(w,z),(x,y,r)]=[(t^{-1}w,t^{\epsilon}z),(t^{a}x,t^{b}y,r)]
\]
for any $t\in T^{1}$
such that $|w|^2+|z|^2=1$ and $|x|^2+|y|^2+r^2=1$.
Define a $T^{3}$-action on $M(\epsilon,a,b)$ by
\[
[(w,z),(x,y,r)]\mapsto [(t_{1}w,z),(t_{2}x,t_{3}y,r)], 
\]
where $(t_{1},t_{2},t_{3})\in T^3$.
Fix an omniorientation on 
$M(\epsilon,a,b)$ by the induced orientations from $S^{3}_{\epsilon}\times S^{4}\subset \C^{2}\times (\C\oplus\C\oplus\R)$.
Then, considering the tangential representations around each fixed point, 
it is easy to check that the induced torus graph is $(\Gamma_{S},\mathcal{A}_{e_{1},e_{2},e_{3}}^{\epsilon,a,b})$, where $e_{1},e_{2},e_{3}$ are the 
standard basis of $\algt_{\Z}\simeq \Z^{3}$.
Therefore, by taking the appropriate automorphism of $T^{3}$, we can construct each torus graph $(\Gamma_{S},\mathcal{A}_{S})$ in Figure \ref{tg-bundles} from $M(\epsilon,a,b)$.
Note that if $\epsilon=-1$ and $a=b$, then this is nothing but one of the torus manifolds 
which appeared in the classifications of torus manifolds with codimension one extended actions in \cite{Ku11}. 

By the argument as above and Corollary \ref{key2}, we establish the following lemma:
\begin{lemma}
\label{S^4-S^2(geom)}
Let $(M,\mathcal{O})$ be an omnioriented $6$-dimensional simply connected torus manifold with $H^{odd}(M)=0$.
If its induced torus graph has $4$-vertices, then
$(M,T)$ is equivariantly diffeomorphic to one of the followings:
\begin{enumerate}
\item $\C P^{3}$ with the standard $T^{3}$-action up to automorphism of $T^{3}$;
\item $M(\epsilon,a,b)$ for some $\epsilon=\pm 1$ and $a,b\in \Z$.
\end{enumerate}
\end{lemma}

\subsection{$6$-dimensional quasitoric manifolds}
\label{sect4.3}

Assume that there are no multiple edges in a $3$-valent torus graph $(\Gamma,\mathcal{A})$, i.e., there are no two edges
that are incident to the same two vertices.
A graph $\Gamma$ is called a {\it simple} if $\Gamma$ does not have both of multiple edges and loops.
In this and Section \ref{sect5}, we study simple torus graphs which can be realized as the one-skeleton of manifold with faces homeomorphic to $D^{3}$.

The typical example of such torus manifolds whose torus graphs are simple is a quasitoric manifold
(introduced by Davis-Januszkiewicz in \cite{DaJa} (also see \cite{BuPa})).
A {\it quasitoric manifold} is defined by 
a torus manifold whose orbit space is a {\it simple convex polytope}, i.e., a convex polytope admitting the structure of a manifold with faces.
For example, the complex projective space $\C P^n$ with the standard $T^n$-action 
is the quasitoric manifold whose orbit space is the $n$-dimensional simplex.
The Figure \ref{tg-CP^3} shows the torus graph induced from $(\C P^3,\mathcal{O}_{\C})$, i.e., the omniorientation $\mathcal{O}_{\C}$ induced from the standard complex structure on $\C P^3$ and the standard $T$-action on $\C P^3$.
\begin{figure}[h]
\begin{center}
\includegraphics[width=150pt,clip]{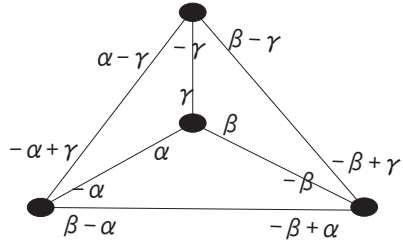}
\caption{Torus graph induced from $(\C P^3,\mathcal{O}_{\C})$.}
\label{tg-CP^3}
\end{center}
\end{figure}

We next characterize when torus graphs are induced from simple convex polytopes, i.e., induced from quasitoric manifolds.
The Steinitz theorem (see \cite[Chapter 4]{Zi}) tells us that
a graph $\Gamma$ is the one skeleton of a $3$-dimensional convex polytope if and only if 
$\Gamma$ is a simple, planner and $3$-connected graph, 
where $\Gamma$ is called a {\it $3$-connected} graph if it remains connected whenever fewer than $3$ vertices are removed.
It easily follows from the Steinitz theorem that we have the following lemma:
\begin{lemma}
\label{polytope}
Let $Q$ be a manifold with faces and $\Gamma$ be its graph.
Assume that $Q$ is homeomorphic to the $3$-disk $D^{3}$ and there are no multiple edges.
Then, the following two statements are equivalent:
\begin{enumerate}
\item $Q$ is combinatorially equivalent to a $3$-dimensional simple convex polytope $P$;
\item $\Gamma$ is a $3$-connected graph.
\end{enumerate}
\end{lemma}

Therefore, 
together with Corollary \ref{key2}, we have the following fact:
\begin{lemma}
\label{quasi-toru}
Let $(M,\mathcal{O})$ be an omnioriented $6$-dimensional simply connected torus manifold with $H^{odd}(M)=0$.
Then, the following two statements are equivalent:
\begin{enumerate}
\item $(M,T)$ is equivariantly diffeomorphic to a quasitoric manifold;
\item its induced torus graph $\Gamma$ is a $3$-connected graph with no multiple edges.
\end{enumerate}
\end{lemma}

\section{Connected sum of torus graphs and other $6$-dimensional torus manifolds}
\label{sect5}

By the arguments in Section \ref{sect4},  only the following case remains:
the simply connected $6$-dimensional torus manifolds with $H^{odd}(M)=0$ 
whose induced torus graphs are simple but not $3$-connected.
Such torus manifolds can be constructed by using the connected sum of ``oriented'' torus graphs.
The purpose of this section is to introduce oriented torus graphs and their connected sum.

We first recall the equivariant connected sum of torus manifolds.
Let $M_{1}$, $M_{2}$ be $2n$-dimensional torus manifolds and $p\in M_{1}^{T}$, $q\in M_{2}^{T}$ be fixed points.
By using the slice theorem, we may take $T$-invariant open neighborhoods $U_{1}\subset M_{1}$ of $p$ and $U_{2}\subset M_{2}$ of $q$.
Assume that $U_{1}$ and $U_{2}$ are equivariantly diffeomorphic.
Then, $U_{1}\setminus\{p\}$ and $U_{2}\setminus\{q\}$ are equivariantly diffeomorphic to $S^{2n-1}\times I$, 
where $S^{2n-1}\subset \C^{n}$ 
with some effective $T^{n}$-action and $I=(-\epsilon, \epsilon)$ with the trivial $T^{n}$-action for some $\epsilon>0$.
Then, we glues these two neighborhood by $\varphi$ defined by the identity on $S^{2n-1}$ and the map $r\mapsto -r$ on $I$ for $r\in I$.
Namely, we can glue $M_{1}\setminus \{p\}$ and $M_{2}\setminus \{q\}$ by the following identification:
\begin{eqnarray}
\label{map-orie}
M_{1}\setminus\{p\}\supset U_{1}\setminus\{p\}\stackrel{\simeq}{\longrightarrow} S^{2n-1}\times I \stackrel{\varphi}{\longrightarrow} 
S^{2n-1}\times I \stackrel{\simeq}{\longrightarrow}  U_{2}\setminus\{q\}\subset M_{2}\setminus\{q\}.
\end{eqnarray}
The $T^{n}$-manifold obtained by this way is denoted by $M_{1} \# M_{2}$ or $M_{1} \#_{(p,q)} M_{2}$ (if we emphasize fixed points $p\in M_{1}^{T}$ and 
$q\in M_{2}^{T}$).
Because each torus manifold has more than two fixed points, $M_{1} \# M_{2}$ is again a torus manifold.
We call this operation an {\it equivariant connected sum}.
The following lemma holds:
\begin{lemma}
\label{lem-easy}
If two torus manifolds $M_{1}$ and $M_{2}$ are simply connected and $H^{odd}(M_{1})=H^{odd}(M_{2})=0$,
then $M_{1} \# M_{2}$ is also simply connected and $H^{odd}(M_{1} \# M_{2})=0$.
\end{lemma}
\begin{proof}
It is easy to check the statement by using van-Kampen's theorem and the Mayer-Vietoris exact sequence.
\end{proof}

Assume that $(M_{1},\mathcal{O}_{1})$, $(M_{2},\mathcal{O}_{2})$ are $6$-dimensional omnioriented simply connected torus manifolds with $H^{odd}(M_{1})=H^{odd}(M_{2})=0$.
Let $(\Gamma_{1},\mathcal{A}_{1})$ and $(\Gamma_{2},\mathcal{A}_{2})$ be their induced $3$-valent torus graphs.
Assume that we can glue $p\in  M_{1}^{T}$ and $q\in M_{2}^{T}$ by the connected sum.
Then, by considering the restriction of $\varphi$ in \eqref{map-orie} onto $S^{2n-1}\subset \C^{n}$, i.e., the identity map,
the axial functions around $p\in V(\Gamma_{1})$ and $q\in V(\Gamma_{2})$ must satisfy 
\begin{eqnarray}
\label{con-sum}
\{\mathcal{A}_{1}(e)\ |\ e\in E_{p}(\Gamma_{1})\}=\{\mathcal{A}_{2}(e)\ |\ e\in E_{q}(\Gamma_{2})\}.
\end{eqnarray}
However, at this stage,
torus graphs $(\Gamma_{1},\mathcal{A}_{1})$ and $(\Gamma_{2},\mathcal{A}_{2})$
do not have the information of orientations of $M_{1}$ and $M_{2}$.
Note that, to do connected sum, we need the orientations around $p\in M_{1}^{T}$ and $q\in M_{2}^{T}$.
To encode the orientations around fixed points, we need the following notion:
\begin{definition}
Let $(\Gamma,\mathcal{A})$ be a torus graph.
We call a triple $(\Gamma,\mathcal{A},\sigma)$ with a map $\sigma:V(\Gamma)\to \{+1, -1\}$ 
an {\it oriented torus graph}, if 
$\sigma:V(\Gamma)\to \{+1, -1\}$ satisfies the following condition for all $e\in E(\Gamma)$:
\begin{eqnarray*}
\sigma(\pi_{V}(e))\mathcal{A}(e)=-\sigma(\pi_{V}(\bar{e}))\mathcal{A}(\bar{e}),
\end{eqnarray*}
where $\pi_{V}(e)\in V(\Gamma)$ is the initial vertex of $e\in E(\Gamma)$, i.e., for $e=pq$, $\pi_{V}(e)=p$ and $\pi_{V}(\bar{e})=q$.
We call a map $\sigma$ an {\it orientation} of $(\Gamma,\mathcal{A})$.
\end{definition}
\begin{remark}
Let $(M,\mathcal{O})$ be an omnioriented torus manifold.
The oriented torus graph $(\Gamma,\mathcal{A},\sigma)$ of $(M,\mathcal{O})$ is defined as follows.
Let $p\in M^{T}$. 
Then, there exist exactly $n$ characteristic submanifolds $M_{1},\ \ldots,\ M_{n}$ such that 
$p$ is a connected component of $\cap_{i=1}^{n}M_{i}$.
Now the fixed orientations of $M_{1},\ldots,M_{n}$ determine the decomposition of the tangential representation, i.e.,
$\psi_{p}:T_{p}M\stackrel{\simeq}{\to} V(\alpha_{1})\oplus\cdots\oplus V(\alpha_{n})$ is determined by fixing the orientations of $M_{1},\ldots,M_{n}$.
On the other hand, the orientation of $M$ determines the orientation of $T_{p}M$.
So, we define the map $\sigma:V(\Gamma)=M^{T}\to \{+1,-1\}$ by
\begin{eqnarray*}
\sigma(p)=
\left\{
\begin{array}{ll}
+1 & \text{if $\psi_{p}$ preserves the orientations} \\
-1 & \text{if $\psi_{p}$ reverses the orientations}
\end{array}
\right.
\end{eqnarray*}
\end{remark}

Let $(\Gamma_{1},\mathcal{A}_{1},\sigma_{1})$ and $(\Gamma_{2},\mathcal{A}_{2},\sigma_{2})$ be the induced oriented torus graphs
from $(M_{1},\mathcal{O}_{1})$ and $(M_{2},\mathcal{O}_{2})$.
If we can glue $p\in  M_{1}^{T}$ and $q\in M_{2}^{T}$ by the connected sum, 
then both of the relation \eqref{con-sum} and 
the relation
\begin{eqnarray}
\label{con-sum2}
\sigma_{1}(p)\not=\sigma_{2}(q)
\end{eqnarray}
hold (\eqref{con-sum2} corresponds to that the orientations on $T_{p}M_{1}$ and $T_{q}M_{2}$ are different).
Note the induced (oriented) torus graph by $M_{1}\#_{(p,q)}M_{2}$ is nothing but the one-skeleton of the connected sum $Q_{1}\#_{(p,q)}Q_{2}$
of manifolds with faces, where $Q_{i}$ is the orbit space of $M_{i}$, $i=1,2$ (also see \cite[Definition 3]{Iz}, \cite[Section 3.1]{Ku10} for details about the connected sum of polytopes).
Therefore, conversely, 
if $p\in V(\Gamma_{1})$ and $q\in V(\Gamma_{2})$ satisfy \eqref{con-sum} and \eqref{con-sum2},
then we can do the {\it connected sum of (oriented) torus graphs} between $(\Gamma_{1},\mathcal{A}_{1},\sigma_{1})$ and $(\Gamma_{2},\mathcal{A}_{2},\sigma_{2})$, say 
$(\Gamma,\mathcal{A},\sigma)=(\Gamma_{1},\mathcal{A}_{1},\sigma_{1})\# (\Gamma_{2},\mathcal{A}_{2},\sigma_{2})$ or $(\Gamma_{1},\mathcal{A}_{1},\sigma_{1})\#_{(p,q)}(\Gamma_{2},\mathcal{A}_{2},\sigma_{2})$ 
(if we emphasize the vertices $p\in V(\Gamma_{1})$ and $q\in V(\Gamma_{2})$).
More precisely, 
$(\Gamma,\mathcal{A},\sigma)=(\Gamma_{1},\mathcal{A}_{1},\sigma_{1})\# (\Gamma_{2},\mathcal{A}_{2},\sigma_{2})$ is defined 
 as follows (also see Figure \ref{operations1}).
\begin{enumerate}
\item $V(\Gamma)=V(\Gamma_{1})\setminus \{p\}\sqcup V(\Gamma_{2})\setminus \{q\}$;
\item $E(\Gamma)=(E(\Gamma_{1})\setminus \{pp_{1},pp_{2},pp_{3}\}) \sqcup (E(\Gamma_{2})\setminus \{qq_{1},qq_{2},qq_{3}\}) 
\sqcup \{p_{1}q_{1},p_{2}q_{2},p_{3}q_{3}\}$, 
where $\mathcal{A}_{1}(pp_{i})=\mathcal{A}_{2}(qq_{i})$ for $i=1,2,3$;
\item $\mathcal{A}:E(\Gamma)\to (\algt^{3}_{\Z})^{*}$ such that $\mathcal{A}(e)=\mathcal{A}_{1}(e)$ and $\mathcal{A}(f)=\mathcal{A}_{2}(f)$ for 
$e\in E(\Gamma_{1})\setminus \{pp_{1},pp_{2},pp_{3}\}$ and $f\in E(\Gamma_{2})\setminus \{qq_{1},qq_{2},qq_{3}\}$,
and $\mathcal{A}(p_{i}q_{i})=\mathcal{A}_{1}(p_{i}p)$ and $\mathcal{A}(q_{i}p_{i})=\mathcal{A}_{2}(q_{i}q)$;
\item $\sigma:V(\Gamma)\to \{+1,-1\}$ is defined by $\sigma(r)=\sigma_{1}(r)$ for $r\in V(\Gamma_{1})\setminus \{p\}$ and 
$\sigma(r')=\sigma_{2}(r')$ for $r'\in V(\Gamma_{2})\setminus \{q\}$.
\end{enumerate}
\begin{figure}[h]
\includegraphics[width=250pt,clip]{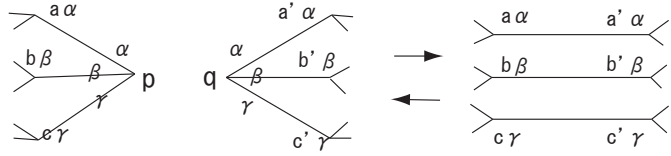}
\caption{This is the local figure of the connected sum $\#_{(p,q)}$ of torus manifold (from left to right) and 
its inverse $\#_{(p,q)}^{-1}$ (from right to left), where $\sigma_{1}(p)\not=\sigma_{2}(q)$.
Here, $\alpha,\beta,\gamma$ are $\Z$-basis of $(\algt_{\Z}^{3})^{*}$ and $a,a',b,b',c,c'=\pm 1$.}
\label{operations1}
\end{figure} 

Then, we can easily check that $(\Gamma,\mathcal{A},\sigma)$ is an oriented torus graph. 
Using Corollary \ref{key2} and the arguments as above, we have the following lemma:
\begin{lemma}
\label{con-sum canonical model}
Let $M_{1}$, $M_{2}$ be $6$-dimensional simply connected torus manifolds with $H^{odd}(M_{1})=H^{odd}(M_{2})=0$,
and $(\Gamma_{1},\mathcal{A}_{1},\sigma_{1})$, $(\Gamma_{2},\mathcal{A}_{2},\sigma_{2})$ be their induced oriented torus graphs from some omniorientations, respectively.
If $(\Gamma,\mathcal{A},\sigma)=(\Gamma_{1},\mathcal{A}_{1},\sigma_{1})\#_{(p,q)}(\Gamma_{2},\mathcal{A}_{2},\sigma_{2})$, then
$(\Gamma,\mathcal{A},\sigma)$ is the oriented torus graph induced from $M=M_{1}\#_{(p,q)}M_{2}$ with some omniorientation.
\end{lemma}

By using the connected sum, we can construct the torus manifolds which are not  
appeared in Section \ref{sect4}.
For example, the following torus manifold is one of such examples:   
\[
\C P^3\# (S^{2}\times S^{4})\# \overline{\C P^3},
\]
where $\overline{\C P^{3}}$ is the reversed orientation of $\C P^{3}$.
The Figure \ref{others} shows the torus graph induced from $\C P^3\# (S^{2}\times S^{4})\# \overline{\C P^3}$ 
(see the axial functions on Figure \ref{tg-bundles}, \ref{tg-CP^3} for details). We can easily check that this graph is $3$-valent, simple and planner but not $3$-connected; therefore, due to Lemma \ref{quasi-toru},
this manifold is not a quasitoric manifold. 
\begin{figure}[h]
\begin{center}
\includegraphics[width=250pt,clip]{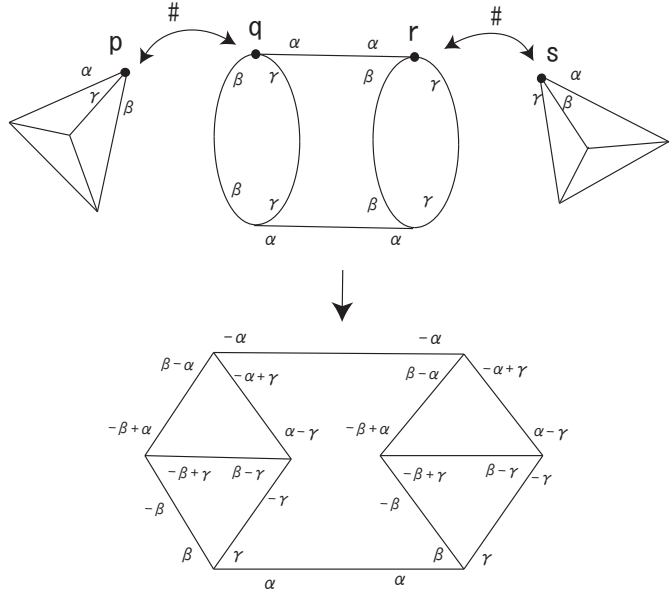}
\caption{Torus graph (with appropriate orientations, e.g, $\sigma(p)=+1,\ \sigma(q)=-1,\ \sigma(r)=+1,\ \sigma(s)=-1$) induced from $\C P^3\# (S^{2}\times S^{4})\# \overline{\C P^3}$.}
\label{others}
\end{center}
\end{figure}

\section{Some combinatorial lemmas}
\label{sect6}

To prove the main theorem (Theorem \ref{main-1}),
in this section, we show the following two lemmas. 
\begin{lemma}
\label{one-pt}
Let $Q$ be a $3$-dimensional manifold with faces which is homeomorphic to $D^{3}$ and $\Gamma$ be its graph.
Then, $\Gamma\setminus \{p\}$ is connected, for all vertices $p\in V(\Gamma)$.
\end{lemma}
\begin{proof}
Because $Q$ is homeomorphic to the $3$-disk $D^{3}$, $\Gamma$ may be regarded as a planner graph by the stereographic projection of $\partial Q=S^{2}$.
Assume $\Gamma\setminus \{p\}$ is not connected. 
Because $Q$ is a $3$-dimensional manifold with faces,
there are exactly three out-going edges from $p$, say $pp_{1}$, $pp_{2}$ and $pp_{3}$.
Therefore, we may assume that there exists a connected component $\Gamma_{1}$ in $\Gamma\setminus \{p\}$ such that 
$p_{1}\in V(\Gamma_{1})$ but $p_{2},p_{3}\not\in V(\Gamma_{1})$ (see Figure \ref{one-pt-con}).
\begin{figure}[h]
\begin{center}
\includegraphics[width=100pt,clip]{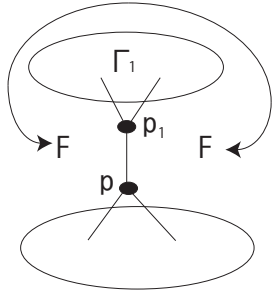}
\caption{The figure explained the proof of Lemma \ref{one-pt}. The facet $F$ has the self-intersection on the edge $p_{1}p$.}
\label{one-pt-con}
\end{center}
\end{figure}
Since $\Gamma_{1}$ is also a planner $3$-valent graph excepts on the vertex $p_{1}$ (because $p\not\in V(\Gamma_{1})$), 
there is a $2$-valent subgraph in $\Gamma_{1}$, say $\partial \Gamma_{1}$, such that $\partial \Gamma_{1}$ splits $\partial Q=S^{2}$ into 
two connected components $H_{+}$ and $H_{-}$, 
where $\Gamma_{1}\setminus \partial \Gamma_{1}\subset H_{+}\setminus \partial \Gamma_{1}$ but $\Gamma_{1}\not\subset H_{-}$. 
This implies that there is a facet $F$ in $Q$ such that $\partial F$ contains $\partial \Gamma_{1}$ and  $p_{1}p$.
However, in this case, $p_{1}p$ must be a self-intersection edge of $F$ (see Figure \ref{one-pt-con}).
This gives a contradiction to that $Q$ is a manifold with faces.
This establishes the statement.
\end{proof}

By Lemma \ref{one-pt}, if $\Gamma$ is not $3$-connected, then
there are two vertices $p,\ q\in V(\Gamma)$ such that $\Gamma\setminus \{p,\ q\}$ is not connected but both of 
$\Gamma\setminus \{p\}$ and $\Gamma\setminus \{q\}$ are connected.
More precisely, we have the following lemma:

\begin{lemma}
\label{separate-points}
Let $Q$ be a $3$-dimensional manifold with faces which is homeomorphic to $D^{3}$ and $\Gamma$ be its graph.
Assume that there are two vertices $p,\ q\in V(\Gamma)$ such that 
$\{p,\ q\}\not\subset V(F)$ for any facets $F$, i.e, 
$p$ and $q$ are not on the same facet $F$.
Then, $\Gamma\setminus \{p,\ q\}$ is connected.
\end{lemma}
\begin{proof}
Assume that $p,\ q$ are not on the same facet of $Q$.
Because $Q$ is a manifold with faces, 
there are mutually distinct facets $F_{1},\ldots,F_{6}$ such that $\{p\}$ is a component of $F_{1}\cap F_{2}\cap F_{3}$ and $\{q\}$ is that of $F_{4}\cap F_{5}\cap F_{6}$ and we can take vertices $p_{1},p_{2},p_{3}$ and $q_{1},q_{2},q_{3}$ 
such that $pp_{i}$ and $qq_{i}$ are all outgoing edges from $p$ and $q$, for $i=1,2,3$.
Take two vertices $r,\ s$ from $\Gamma\setminus \{p,q\}$.
By Lemma \ref{one-pt}, $\Gamma\setminus\{q\}$ is connected. 
So there is a path $\gamma$ from $r$ to $s$ in $\Gamma\setminus\{q\}$.
If $\gamma$ does not go through $p$, then $r$ and $s$ are connected in $\Gamma\setminus \{p,q\}$.
Assume that this path $\gamma$ goes through $p$. 
Then $\gamma$ goes through exactly two vertices $p_{i}$, $p_{j}$ (we may assume $p_{1}$ and $p_{2}$).
Moreover, one of the facets $F_{1},F_{2},F_{3}$, say $F_{1}$, contains both of $p_{1}$, $p_{2}$. 
Note that $F_{1}$ corresponds to the $2$-valent subgraph in $\Gamma$. 
Therefore, we can take the path $\gamma_{p}$ connecting $p_{1}$, $p_{2}$ on $F_{1}$ which is not the path $p_{1}pp_{2}$.
Because $p,\ q$ are not on the same facet, in particular $q\not\in V(F_{1})$, the path $\gamma_{p}$ does not contain $q$.
Hence, the connected subgraph $\gamma\cup \gamma_{p}$ contains both of $r,s$ but does not contain both of $p,q$.
Thus, we can take the path $\gamma'$  from $r$ to $s$ in $\gamma\cup \gamma_{p}\subset \Gamma\setminus \{p,\ q\}$.
This establishes that $\Gamma\setminus \{p,q\}$ is connected.
\end{proof}

In summary, by Lemma \ref{one-pt} and \ref{separate-points}, we have the following fact.
\begin{corollary}
\label{key3}
Let $\Gamma$ be a one-skeleton of $3$-dimensional manifold with faces $Q$.
Then, for all $p\in V(\Gamma)$, $\Gamma\setminus \{p\}$ is connected.
Furthermore, if $\Gamma\setminus \{p,\ q\}$ is not connected, then $p,q$ is on the same facet.
\end{corollary}

\section{Proof of main theorem}
\label{sect7}

The main theorem of this paper can be stated as follows:
\begin{theorem}
\label{main-1}
Let $M$ be a simply connected, $6$-dimensional torus manifold with $H^{odd}(M)=0$.
Then, $M$ is equivariantly diffeomorphic to one of the following manifolds:
\begin{enumerate}
\item $S^6\subset \C^{3}\oplus \R$ with a torus action induced from a (faithful) representation of $T^{3}$ on $\C^{3}$;
\item $6$-dimensional quasitoric manifold $X$; 
\item $S^4$-bundle over $S^2$ which is equivariantly diffeomorphic to $M(\epsilon,a,b)$ for some $\epsilon=\pm 1,\ a,\ b\in \Z$,
\end{enumerate}
or otherwise, there are some $6$-dimensional quasitoric manifolds $X_{h}$, for some $h=1,\ldots,k$, and some $S^4$-bundles over $S^2$, 
say $S_{i}=M(\epsilon_{i},a_{i},b_{i})$ (for some $\epsilon_{i}=\pm 1$, $a_{i},b_{i}\in \Z$ and $i=1,\ldots,\ell$), such that
$M$ is equivariantly diffeomorphic to
\[
(\#_{h=1}^{k}X_{h})\# (\#_{i=1}^{\ell} S_{i})
\]
where $\#$ represents the equivariant connected sum around fixed points, $k+\ell\ge 2$ for $k\ge 0,\ \ell\ge 1$,
and the case $k=0$ means that 
there is no $X_{h}$ factor.
\end{theorem}

In this final section, we prove Theorem \ref{main-1}.

Let $M$ be 
a simply connected, $6$-dimensional torus manifold with $H^{odd}(M)=0$, 
$Q$ be its orbit space which is homeomorphic to $D^{3}$ and
$(\Gamma_{M},\mathcal{A}_{M})$ be its induced oriented torus graph (we omit the orientation).

Because $\Gamma_{M}$ is a one-skeleton of a manifold with faces which is homeomorphic to $D^{3}$, 
it is easy to check that $|V(\Gamma_{M})|\not=1, 3$.
If $|V(\Gamma_{M})|=2$,
by Lemma \ref{S^6}, we have that 
$M$ is equivariantly diffeomorphic to $S^6$, i.e., the statement (1).
If $|V(\Gamma_{M})|= 4$, it follows from Lemma \ref{S^4-S^2(geom)} that 
$M$ is equivariantly diffeomorphic to a quasitoric manifold $\C P^3$ or $M(\epsilon,a,b)$ for some $\epsilon=\pm 1$, $a,b\in \Z$,
i.e., the statement (2) or (3) occurs.
So we may only prove the case when $|V(\Gamma_{M})|\ge 5$. 

We first claim the following lemma:
\begin{lemma}
\label{1-claim}
Assume that $|V(\Gamma_{M})|\ge 5$ and there is a multiple edge in $\Gamma_{M}$.
Then, $(\Gamma_{M},\mathcal{A}_{M})$ can be decomposed into the connected sum of 
the following torus graphs:
\[
(\Gamma_{M},\mathcal{A}_{M})=(\Gamma_{X},\mathcal{A}_{X})\# (\Gamma_{S_{1}},\mathcal{A}_{S_{1}})\#\cdots \# (\Gamma_{S_{\ell'}},\mathcal{A}_{S_{\ell'}})
\]
or 
\[
(\Gamma_{M},\mathcal{A}_{M})=(\Gamma_{S_{1}},\mathcal{A}_{S_{1}})\#\cdots \# (\Gamma_{S_{\ell'}},\mathcal{A}_{S_{\ell'}}),
\]
 where $(\Gamma_{X},\mathcal{A}_{X})$ is a torus graph without multiple edges and $S_{i}=M(\epsilon_{i},a_{i},b_{i})$ for $i=1,\ldots,\ell'$.
\end{lemma}
\begin{proof}
Assume two vertices $p$, $q$ are connected by a multiple edge, i.e., two edges (see the bottom graph in Figure \ref{essential discussion to prove this theorem}).

\begin{figure}[h]
\begin{center}
\includegraphics[width=250pt,clip]{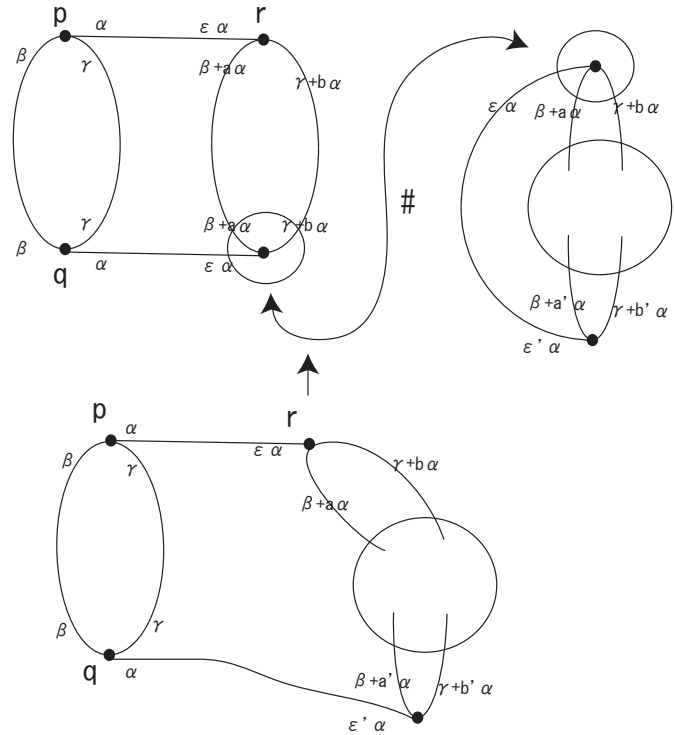}
\caption{We may regard $\alpha,\beta,\gamma$ as any generators in $(\algt_{\Z}^{3})^{*}$ and $a,a',b,b'\in \Z$ and $\epsilon,\epsilon'=\pm 1$.
Here, the bottom graph is $(\Gamma_{M},\mathcal{A}_{M})$ and the left upper graph is $(\Gamma_{S_{1}},\mathcal{A}_{S_{1}})$ 
and the right upper graph is $(\Gamma_{M'},\mathcal{A}_{M'})$.
Note that if we fix the orientation of $(\Gamma_{M},\mathcal{A}_{M})$ then the orientations of $(\Gamma_{S_{1}},\mathcal{A}_{S_{1}})$ and 
$(\Gamma_{M'},\mathcal{A}_{M'})$ are automatically determined.}
\label{essential discussion to prove this theorem}
\end{center}
\end{figure}

Then, by the connection of the torus graph (see Proposition \ref{connection}), 
it is easy to check that the axial functions around 
the vertex $r$ of the bottom graph in Figure \ref{essential discussion to prove this theorem}
satisfy the axial functions expressed in 
Figure \ref{essential discussion to prove this theorem},
where we can take $\alpha$, $\beta$, $\gamma$ as $\Z$-basis $(\algt_{\Z}^{3})^{*}$.
In this case, we can do (inverse) connected sum 
such as expressed 
 in Figure \ref{essential discussion to prove this theorem} (from the bottom to the top in Figure \ref{essential discussion to prove this theorem}).
Then, the induced torus graph $(\Gamma_{M},\mathcal{A}_{M})$ is decomposed into
two induced torus graphs 
$(\Gamma_{S_{1}},\mathcal{A}_{S_{1}})$ and $(\Gamma_{M'},\mathcal{A}_{M'})$, where $M'$ is some simply connected $6$-dimensional torus manifold with $H^{odd}(M')=0$ by Lemma \ref{lem-easy}.
Namely, we have
\[
(\Gamma_{M},\mathcal{A}_{M})= (\Gamma_{M'},\mathcal{A}_{M'})\# (\Gamma_{S_{1}},\mathcal{A}_{S_{1}}).
\]
If there is no multiple edges in $\Gamma_{M'}$, then we may put $\Gamma_{M'}=\Gamma_{X}$.
Assume that there is a multiple edge in $\Gamma_{M'}$. 
If there are only $4$ vertices in $\Gamma_{M'}$, then we may put $M'$ as $S_{2}=M(\epsilon_{2},a_{2},b_{2})$ by Lemma \ref{S^4-S^2(geom)}.
When there are more than $4$ vertices in $\Gamma_{M'}$,
with iterating the arguments as above, finally 
we establish the statement of this lemma.
\end{proof}

Therefore, to prove Theorem \ref{main-1}, it is enough to show the following lemma:
\begin{lemma}
\label{2-claim}
Assume that $|V(\Gamma_{M})|\ge 5$ and there are no multiple edges in $\Gamma_{M}$.
Then, $(\Gamma_{M},\mathcal{A}_{M})$ can be decomposed into the connected sum of 
the following torus graphs:
\[
(\Gamma_{M},\mathcal{A}_{M})=(\Gamma_{X_{1}},\mathcal{A}_{X_{1}})\# \cdots \# (\Gamma_{X_{k}},\mathcal{A}_{X_{k}})\#
(\Gamma_{S_{1}},\mathcal{A}_{S_{1}})\#\cdots \# (\Gamma_{S_{\ell''}},\mathcal{A}_{S_{\ell''}})
\]
 where $(\Gamma_{X_{h}},\mathcal{A}_{X_{h}})$ for $h=1,\ldots,k$ is the torus graph induced from a quasitoric manifold $X_{h}$, and $S_{i}=M(\epsilon_{i},a_{i},b_{i})$ for $i=1,\ldots,\ell''$.
\end{lemma}
\begin{proof}
If $\Gamma_{M}(=\Gamma)$ is $3$-connected, then it follows from Lemma \ref{quasi-toru} that the statement holds, i.e., $k=1,\ell''=0$.
Therefore, we may assume $\Gamma$ is not $3$-connected.
In this case,   
by Corollary \ref{key3}, there is a $2$-valent torus subgraph $F\subset \Gamma$ such that 
some $p,q\in V(F)$ satisfy that $\Gamma\setminus \{p,q\}$ is not connected.

If $F$ is a triangle (i.e., $|V(F)|= 3$), 
with the method similar to that demonstrated in the proof of Lemma \ref{one-pt}, 
we have that there is a face in $Q$ which has the self-intersection edge.
This gives a contradiction to that $Q$ is a manifold with faces.
Therefore, we may assume $|V(F)|\ge 4$.
We first assume that $pq$ is an edge of $F$.
Then, there are two graphs $\Gamma_{1}$ and $\Gamma_{2}$ which are the connected components of $\Gamma\setminus \{p,q\}$ expressed  
in Figure \ref{edge-con}.
\begin{figure}[h]
\begin{center}
\includegraphics[width=70pt,clip]{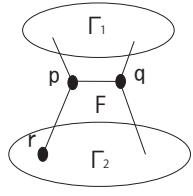}
\caption{If we remove $r$ and $q$ from $\Gamma$ instead of $p$ and $q$, the graph is also disconnected.}
\label{edge-con}
\end{center}
\end{figure}
If we remove the two vertices $r$ and $q$ from $\Gamma$ instead of $p$, $q$, where $r\in V(\Gamma_{2})$ such that $pr$ is an edge, 
then $\Gamma\setminus \{r,q\}$ is also not connected (see Figure \ref{edge-con}). 
Therefore, from now, we may assume the followings:
\begin{enumerate}
\item $p,q\in V(\Gamma)$ satisfy $\Gamma\setminus \{p,q\}$ is not connected;
\item $pq\not\in E(\Gamma)$;
\item there is a $2$-valent torus subgraph (facet) $F$ with $|V(F)|\ge 4$ in $\Gamma$ such that $p,q\in V(F)$.
\end{enumerate}
We call such facet $F$ a {\it singular facet}. 

Let $F$ be a singular facet.
Assume $|V(F)|\ge 6$.
\begin{figure}[h]
\begin{center}
\includegraphics[width=120pt,clip]{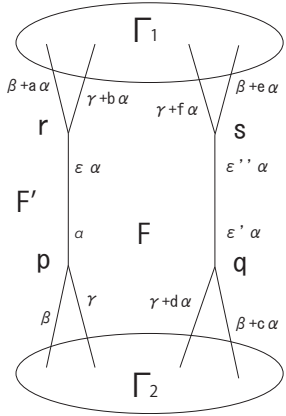}
\caption{The axial functions around $F$ when $|V(F)|\ge 6$, where $\epsilon,\ \epsilon',\ \epsilon''=\pm 1$ and $a,b,c,d,e,f\in \Z$.
Here, $F'$ is a facet which intersects with $F$ on $pr$ and $qs$.}
\label{6-vertices}
\end{center}
\end{figure}
In this case, by the similar argument just before, we may take $p,\ q$ as in the position of Figure \ref{6-vertices}, i.e.,
$p$ and $q$ are on two separated edges $rp$ and $sq$ which are common edges of two facets $F$ and $F'$ in Figure \ref{6-vertices}
(Note that $r$ and $s$ might be connected by an edge).
Moreover, by considering the omnioriented characteristic functions of the facets $F$ and $F'$, 
we may take the axial functions around the facet $F$ as in Figure \ref{6-vertices}.

By taking an appropriate orientation, we can do the connected sum as in Figure \ref{6-vertices2}; here
we denote the (oriented) torus graph containing $\Gamma_{1}$ as $(\widetilde{\Gamma}_{1},\widetilde{\mathcal{A}}_{1})$ and 
that containing $\Gamma_{2}$ as $(\widetilde{\Gamma}'_{2},\widetilde{\mathcal{A}}'_{2})$.
\begin{figure}[h]
\begin{center}
\includegraphics[width=120pt,clip]{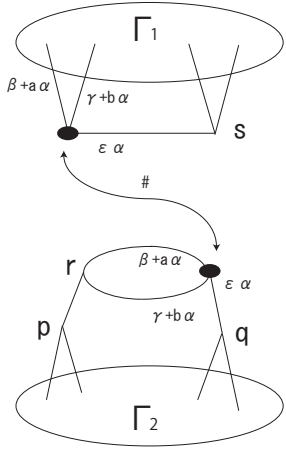}
\caption{The torus graph $(\Gamma,\mathcal{A})$ in Figure \ref{6-vertices} splits into two torus graphs $(\widetilde{\Gamma}_{1},\widetilde{\mathcal{A}}_{1})$ (upper) and $(\widetilde{\Gamma}'_{2},\widetilde{\mathcal{A}}'_{2})$ (lower).
Here, we omit the axial functions around the vertices $p,q,r,s$ because they are exactly same with that of Figure \ref{6-vertices}.}
\label{6-vertices2}
\end{center}
\end{figure}

The torus graph obtained by this connected sum is nothing but the torus graph $(\Gamma,\mathcal{A})$ in Figure \ref{6-vertices}.
Note that $\widetilde{\Gamma}_{1}$ is simple and planner, and $\widetilde{\Gamma}'_{2}$ is just planner.
With the method similar to that demonstrated in Figure \ref{essential discussion to prove this theorem}, 
$(\widetilde{\Gamma}'_{2},\widetilde{\mathcal{A}}'_{2})$ can be obtained from the connected sum of $(\Gamma_{S},\mathcal{A}_{S})$
and the simple, planner graph $(\widetilde{\Gamma}_{2},\widetilde{\mathcal{A}}_{2})$ (containing $\Gamma_{2}$), where 
$(\Gamma_{S},\mathcal{A}_{S})$ is one of the torus graphs (by taking the appropriate axial functions) in Figure \ref{tg-bundles}.
Namely, the torus graph in Figure \ref{6-vertices} can be obtained from the following connected sum:
\[
(\Gamma,\mathcal{A})=(\widetilde{\Gamma}_{1},\widetilde{\mathcal{A}}_{1})\# (\Gamma_{S},\mathcal{A}_{S})\# (\widetilde{\Gamma}_{2},\widetilde{\mathcal{A}}_{2}).
\]
Here, it is easy to check that 
$\widetilde{\Gamma}_{i}$ consists of $\Gamma_{i}$ and the other two facets, say $\widetilde{F}(i)$ and $\widetilde{F}'(i)$ (induced from $F$ and $F'$ in $\Gamma$).
Because of Figure \ref{6-vertices2}, the number of vertices of other two facets $\widetilde{F}(i)$ and $\widetilde{F}'(i)$ are reduced; in particular, 
the number of vertices of the facet $\widetilde{F}(i)$ induced from the singular facet $F$ is strictly less than $6$.
If both of $(\widetilde{\Gamma}_{1},\widetilde{\mathcal{A}}_{1})$ and $(\widetilde{\Gamma}_{2},\widetilde{\mathcal{A}}_{2})$ are $3$-connected, then
these torus graphs are induced from quasitoric manifolds, i.e, 
the statements of Lemma \ref{2-claim} hold.
Assume that $(\widetilde{\Gamma}_{1},\widetilde{\mathcal{A}}_{1})$ is not $3$-connected.
Then, by the arguments before, there is a singular facet $F$ in $(\widetilde{\Gamma}_{1},\widetilde{\mathcal{A}}_{1})$.
If $|V(F)|\ge 6$, then $(\widetilde{\Gamma}_{1},\widetilde{\mathcal{A}}_{1})$ also decomposes into 
\[
(\widetilde{\Gamma}_{1},\widetilde{\mathcal{A}}_{1})=(\widetilde{\Gamma}_{3},\widetilde{\mathcal{A}}_{3})\# (\Gamma_{S'},\mathcal{A}_{S'})\# (\widetilde{\Gamma}_{4},\widetilde{\mathcal{A}}_{4}),
\]
by using the similar arguments explained in Figure \ref{6-vertices2}. 
Iterating this arguments, we may reduce all singular facets with $|V(F)|\ge 6$.
More precisely, 
we may decompose $(\Gamma,\mathcal{A})$ in Figure \ref{6-vertices} into
\[
(\Gamma,\mathcal{A})=
\#_{i=1}^{\ell}\left\{
(\Gamma_{i},\mathcal{A}_{i})\# (\Gamma_{S_{i}},\mathcal{A}_{S_{i}})\# (\Gamma_{i+\ell},\mathcal{A}_{i+\ell})
\right\},
\]
where $(\Gamma_{S_{i}},\mathcal{A}_{S_{i}})$, for $i=1,\ldots,\ell$, is a torus graph in Figure \ref{tg-bundles} and 
$(\Gamma_{h},\mathcal{A}_{h})$, for $h=1,\ldots, 2\ell$, is a $3$-valent simple and planner torus graph which satisfies one of the followings:
\begin{itemize}
\item $3$-connected (in this case, $(\Gamma_{h},\mathcal{A}_{h})$ is induced from a quasitoric manifold);
\item otherwise, all singular facets $F$ satisfy $|V(F)|=4$ or $5$.
\end{itemize}

Assume that the number of vertices in every singular facet of the torus graph $(\Gamma,\mathcal{A})$ is less than or equal to $5$.
Then, such torus graph is one of torus graphs expressed in Figure \ref{45-vertices}.
\begin{figure}[h]
\begin{center}
\includegraphics[width=200pt,clip]{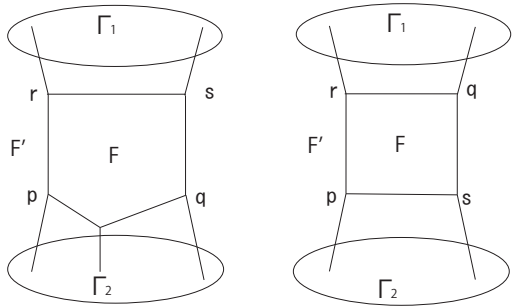}
\caption{The singular facets $F$ with $|V(F)|=5$ or $4$.
Here, $F'$ is a facet which intersects with $F$ on $pr$ and $qs$.}
\label{45-vertices}
\end{center}
\end{figure}
However, because $\Gamma$ is the one-skeleton of a manifold with faces and not $3$-connected, it is easy to check that there exists the singular facet $F'$ 
such that $F'\cap F=\{pr,qs\}$ and $|V(F')|\ge 6$.
This gives a contradiction.
Hence, this case does not occur.
This establishes Lemma \ref{2-claim}.
\end{proof}

Consequently, 
by Lemma \ref{con-sum canonical model}, \ref{1-claim} and \ref{2-claim}, 
we have the statement of Theorem \ref{main-1}

Finally, 
by Theorem \ref{main-1} and the Mayer-Vietoris exact sequence, 
we also have the following well-known result.
\begin{corollary}
\label{2nd deg}
Let $M$ be a simply connected $6$-dimensional torus manifold whose cohomology ring is generated by the $2$nd degree cohomology.
Then, $M$ is a quasitoric manifold.
\end{corollary}

\section*{Acknowledgments}
The author would like to thank 
Professors Yael Karshon and Takashi Tsuboi for providing him excellent circumstances to do research.



\end{document}